\RequirePackage{fix-cm}
\documentclass[dfraft,smallextended]{svjour3}       
\smartqed  
\usepackage{graphicx}
\usepackage[margin=1in]{geometry}  
\usepackage{amsmath}               
\usepackage{amsfonts}              
\usepackage{amsthm}                
\usepackage[utf8]{inputenc}
\usepackage{easybmat}

\begin{document}

\title{An additional structure over integer rings $\mathbb{Z}_{p^r}^n$ }
\markright{An additional structure over integer rings}
\author{A. Cambraia Jr, A. O. Moura and A. T. Silva}

%
%
%
%
%
%

\institute{A. Cambraia Jr, A. O. Moura and A. T. Silva \at
              Departamento de Matemática, Universidade Federal de Viçosa, Av. P. H. Rolfs, s/n, Campus Universitário, Viçosa, Minas Gerais, Brazil. \\
              \email{ady.cambraia@ufv.br, allan.moura@ufv.br, anderson.tiago@ufv.br}           
}

\date{Received: date / Accepted: date}

\maketitle

\begin{abstract}
We present an algebraic structure in modules over integer rings with cardinality prime powers, which allows to define bases. With such structure, we prove a similar version for the basis extension theorem of linear algebra over fields. Moreover, we exhibit results involving the modules and their duals.

\end{abstract}

%

\section{Introduction.}

The vector spaces over fields are very important in the research on several scientific areas, due to their wide application, as can
be observed in Codes Theory, Computational Modeling, Cryptography,
Quantum Mechanics, etc.

These sets are important due to their algebraic structure, which allows
obtaining bases, writing any element in the set as an unique linear combination
of the basis, and obtaining results over dimension and the dual set
and its properties.

Modules over rings are an algebraic structure that is a natural extension of the vector
spaces over fields. The free modules are a particular example
of modules over rings. Such sets admit a similar
algebraic structure as the vectors spaces. It is still possible to obtain bases, to determine dimensions and so on. However, most modules over rings are not free modules. Recent researches,
for example in the Error Correcting Codes Theory, has shown that
these concepts, related to vectors spaces, should be extended to modules. Good codes that are submodules of the $\mathbb{Z}_{4}^{n}$ were determined in  \cite{hammons}, \cite{conway} and  \cite{pless}. However, it is very difficult
to work with such set over rings.
So, concepts for any module over $\mathbb{Z}_{p^r}^{n}$ with $p$ prime, to obtain a structure
of the vector spaces to such sets, were introduced in \cite{calderbank} and \cite{kuijper}. Such structure has been used by several researchers. However, there is little information about modules with this structure.


In this work we prove the similar theorem of the basis extension of linear algebra over fields, but
for modules over $\mathbb{Z}_{p^r}$ rings. Moreover, we exhibit results involving the modules and their duals.

To illustrate the importance of setting up a structure that allows defining basis and describing elements of unique form, we present the following example. Let $\mathbb{Z}_{8}^{3}$ be the module over the
ring $\mathbb{Z}_{8}$. Let 
\[
\begin{array}{rl}
M=\left\{ (0,0,0),(2,0,1),(4,0,2),(6,0,3),(0,0,4),(2,0,5),(4,0,6),(6,0,7),\right.\\
\left. (0,4,0),(2,4,1),(4,4,2),(6,4,3),(0,4,4),(2,4,5),(4,4,6),(6,4,7)\right\}
\end{array} 
\]
 be a submodule of $\mathbb{Z}_{8}^{3}$ over $\mathbb{Z}_{8}$. We
observe that a minimal generator set for $M$ is $\beta=\left\{ (2,0,1),(0,4,0)\right\} $.
In this case, there is no unicity for the description of the elements
of $M$ as $\mathbb{Z}_{8}$ linear combination of $\beta$. We exemplify,
for the element $(4,4,6)$:

\[
\begin{array}{rl}
(4,4,6)= & 6.(2,0,1)+1.\left(0,4,0\right),\\
(4,4,6)= & 6.(2,0,1)+3.\left(0,4,0\right),\\
(4,4,6)= & 6.(2,0,1)+5.\left(0,4,0\right),\\
(4,4,6)= & 6.(2,0,1)+7.\left(0,4,0\right).
\end{array}
\]

We also observe that $\beta$ is not an independent set because,

$$
\begin{array}{ccl}
0.\left(2,0,1\right)+2\left(0,4,0\right) & = & 0.\left(2,0,1\right)+4\left(0,4,0\right) \\
& = & 0.\left(2,0,1\right)+6\left(0,4,0\right) \\
& = & (0,0,0).
\end{array}
$$

\section{Preliminary.}

Let $\langle \ , \ \rangle : Z_{p^r}^n \times Z_{p^r}^n \longrightarrow Z_{p^r}$ be the {\it symmetric bilinear form} on the space $Z_{p^r}^n$ defined by
$$\langle u , v \rangle =\sum_{i=1}^nx_iy_i,$$ where $u=(x_1,\dots,x_n)$ and $v=(y_1,\dots,y_n).$

Let $S=\left\{ v_{1},v_{2},\dots,v_{k}\right\} $ be $k$ vectors in
$\mathbb{Z}_{p^{r}}^{n}$. The span of $S$ is denoted by $M$, i.e., $M=span(S)$. The {\it dual} of $M$, denoted by $M^\perp$, is given by
$$M^\perp=\left\{v \in  \mathbb{Z}_{p^{r}}^{n}; \langle u,v\rangle=0, \ \forall u \in M\right\}.$$


The vector $\underset{j=1}{\overset{k}{\sum}}a_{j}v_{j}$ is
called a \emph{$p$-linear combination} of $S=\{v_{1},\dots,v_{k}\}$, where \linebreak$a_j\in\mathcal{A}_{p}=\left\{ 0,1,2,\dots,p-1\right\} \subset\mathbb{Z}_{p^{r}}$. The set of all $p$-linear combinations of $S$ is called {\it $p$-span(S)}.

An ordered set $\left(v_{1},v_{2},\dots,v_{k}\right)$ in $\mathbb{Z}_{p^{r}}^{n}$
is a {\it $p$-generator sequence}  if 
\begin{description}
\item [{i)}] $pv_{k}=0$;
\item [{ii)}] for $1\leq i\leq k-1$, the vector $pv_{i}$ is a $p$-linear
combination of $v_{i+1},\dots,v_{k}$.
\end{description}

A family $v_1,v_2, \dots, v_k$ of $\mathbb{Z}_{p^r}^n$ is said to be {\it $p$-linearly independent}, if for \linebreak $a_1, a_2, \dots, a_k\in \mathcal{A}_p$, then $$\sum_{j=1}^ka_jv_j=0 \Leftrightarrow a_j=0, \ \forall j=1\dots k.$$

\begin{theorem}[\cite{vazirani}]\label{teo0}
Let $( v_{1},v_{2},\dots,v_{k}) $ be a $p$-generator sequence in $\mathbb{Z}_{p^r}^n.$ Then $$p{\textrm -span}(v_{1},v_{2},\dots,v_{k}) = {\textrm span}(v_{1},v_{2},\dots,v_{k}).$$ In particular, $p$-span$(v_{1},v_{2},\dots,v_{k})$ is a submodule of $\mathbb{Z}_{p^r}^n.$
\end{theorem}

\begin{example}
Let $(v_1,v_2)$ be a $p$-generator sequence in $\mathbb{Z}_{p^r}^n,$ then \begin{equation}{\label{gerador}}
\begin{array}{l}
pv_1=\gamma v_2, \ \textrm{ with }  \gamma \in \mathcal{A}_p, \\
 \\
  pv_2=0.
\end{array}
\end{equation}

 Let $v \in span(v_1,v_2)$, then $v=av_1+bv_2$ with $a, b \in \mathbb{Z}_{p^r}$. As each element of $\mathbb{Z}_{p^r}$ has a unique $p$-adic decomposition, it follows that $$\begin{array}{l}
 a=a_0+a_1p+\cdots+a_{r-1}p^{r-1} \ \textrm{ and }\\ \\ 
 b=b_0+b_1p+\cdots+p_{r-1}p^{r-1},
 \end{array}$$ where $a_i, b_i \in \mathcal{A}_p$, for $i=0,\cdots,r-1$. Thus, by   \eqref{gerador}
 
 $$\begin{array}{ll}
  av_1&=a_0v_1+a_1pv_1+\cdots+a_{r-1}p^{r-1}v_1\\
  \\
  &=a_0v_1+a_1pv_1=a_0v_1+a_1\gamma v_2, \ \textrm{ and }\\
  \\
bv_2&=b_0v_2+b_1pv_2+\cdots+p_{r-1}p^{r-1}v_2=b_0v_2.
 \end{array}$$
 
Therefore, $$v=(a_0v_1+a_1\gamma v_2)+b_0v_2=a_0v_1+(a_1\gamma+b_0)v_2.$$ Using again the $p$-adic decomposition of $c=a_1\gamma+b_0 \in \mathbb{Z}_{p^r}$ and the equations \eqref{gerador}, it follows that $$c=c_0+c_1p+\cdots+c_{r-1}p^{r-1},$$ where $c_i \in \mathcal{A}_p$. Thus $v=a_0v_1+c_0v_2,$ with $a_0, c_0$ in $\mathcal{A}_p$ and $v \in p$-$span(v_1,v_2).$
\end{example}

A $p$-linearly independent $p$-generator sequence will be called a {\it $p$-basis}. Notice that the $p$-linear combinations of the elements of a $p$-basis$(v_1,\dots, v_k)$ uniquely generate the elements of the submodule $M=p$-span($v_1,\dots, v_k$). Otherwise, the null vector should be a non-trivial $p$-linear combination. Then, $|M|=p^k$, where $|\cdot|$ denotes the number of elements of a set.  We will define the {$p$-dimension} of this submodule as $k$.

Let $M$ be a submodule of $\mathbb{Z}_{p^{r}}^{n}$. Consider 
$$Soc(M)=\left\{v \in  M; pv=0\right\}.$$

A generator matrix  $G$ for a submodule $M$ with length $n$ is said to be in {\it standard form} if,  after a suitable permutation of the coordinates, 

$$G=\begin{bmatrix}
   I_{k_0} & A_{01} & A_{02} & \dots  & A_{0,r-1} & A_{0,r}\\
    0 & \gamma I_{k_1}  & \gamma A_{12} & \dots  & x_{2n} & \gamma A_{1,r} \\
    0 & 0 & \gamma^2I_{k_2} & \dots  & x_{2n} & \gamma^2 A_{2,r} \\
    \vdots & \vdots & \vdots & \ddots & \vdots & \vdots\\
    0 & 0 & 0 & \dots  & \gamma^{r-1}I_{k_{r-1}} & \gamma^{r-1}A_{r-1,r}
\end{bmatrix},
$$ where $I_k$ denotes the $k\times k$ identidy matrix and the columns are grouped into blocks of sizes $k_0, k_1, \cdots, k_{r-1}, n-\sum_{i=0}^{r-1}k_i$ with $k_i \geq 0.$

\begin{theorem}[\cite{norton,norton2}] \label{teo1} Any submodule $M$ of $\mathbb{Z}_{p^r}^n$ has a generator matrix in standard form. All generator matrices in standard form for a submodule $M$ have the same parameters $k_0, \cdots, k_{r-1}$ and $|M|=\left|K\right|^{\sum_{i=0}^{r-1}\left(r-i\right)k_{i}\left(M\right)}$, where $K$ is the residual field of $\mathbb{Z}_{p^r}$. 
\end{theorem}

Given a submodule $M$ of $\mathbb{Z}_{p^r}^n$, the number of rows of a generator matrix $G$ in standard form for $M$ is denoted by $k(M)$ and for $i=0,\cdots,r-1$, $k_i(M)$ denotes the number of rows of $G$ that are divisible by $\gamma^i$ but not by $\gamma^{i+1}.$ Then, $k(M)=\sum_{i=0}^{r-1}k_i(M).$

\begin{theorem}[\cite{norton,norton2}] \label{teo2} Let $M$ be a submodule of $\mathbb{Z}_{p^r}^n$ with generator matrix $G$ in standard form. Then, $k(M^\perp)=n-k_0(M), k_0(M^\perp)=n-k(M),$ and $k_i(M^\perp)=k_{r-i}(M),$ for $i=1, \cdots, r-1.$
\end{theorem}

%
%

\section{Results.}

\begin{lemma}
\label{prop:cardinalidade} Let $M$ be a submodule of $\mathbb{Z}_{p^{r}}^{n}$. Then $$\left|M\right|\left|M^{\perp}\right|=\left|\mathbb{Z}_{p^{r}}^{n}\right|=p^{r\cdot n}.$$
\end{lemma}
\begin{proof}
By Theorem \ref{teo1}, $\mbox{\ensuremath{\left|M\right|}}=\left|K\right|^{\sum_{i=0}^{r-1}\left(r-i\right)k_{i}\left(M\right)}$.
By Theorem \ref{teo2}, $k_{0}\left(M^{\perp}\right)=n-k\left(M\right)$ and
$k_{i}\left(M^{\perp}\right)=k_{r-i}\left(M\right)$ for $i=1,2,\dots,r-1$.
We observe that $k\left(M\right)=\sum_{i=0}^{r-1}k_{i}\left(M\right)$.

Now 

$\begin{array}{rl}
\displaystyle\sum_{i=0}^{r-1}\left(r-i\right)k_{i}\left(M^{\perp}\right) & =\displaystyle\sum_{i=1}^{r-1}\left(r-i\right)k_{i}\left(M^{\perp}\right)+rk_{0}\left(M^{\perp}\right)\\
 & =\displaystyle\sum_{i=1}^{r-1}\left(r-i\right)k_{i}\left(M^{\perp}\right)+r\left(n-k\left(M\right)\right)\\
 & =\displaystyle\sum_{i=1}^{r-1}\left(r-i\right)k_{r-i}\left(M\right)+r\left(n-\displaystyle\sum_{i=0}^{r-1}k_{i}\left(M\right)\right)\\
 & =rn+\displaystyle\sum_{i=1}^{r-1}\left(r-i\right)k_{r-i}\left(M\right)-r\left(\displaystyle\sum_{i=0}^{r-1}k_{i}\left(M\right)\right)\\
 & =rn-\displaystyle\sum_{i=0}^{r-1}\left(r-i\right)k_{i}\left(M\right).
\end{array}$

And those 
\[
\begin{array}{rll}
\left|M^{\perp}\right|&=\left|K\right|^{\sum_{i=0}^{r-1}\left(r-i\right)k_{i}\left(M^{\perp}\right)}&=\left|K\right|^{rn-\sum_{i=0}^{r-1}\left(r-i\right)k_{i}\left(M\right)}\\
&=\dfrac{\left|K\right|^{rn}}{\left|K\right|^{\sum_{i=0}^{r-1}\left(r-i\right)k_{i}\left(M\right)}}&=\dfrac{\left|\mathbb{Z}_{p^{r}}^{n}\right|}{\left|M\right|}.
\end{array}
\]
\end{proof}
\begin{theorem}
Let $M$ be a submodule of $\mathbb{Z}_{p^r}^n$. Then
\[
p\mbox{-}dim\left(M\right)+p\mbox{-}dim\left(M^{\perp}\right)=p\mbox{-}dim\left(\mathbb{Z}_{p^{r}}^{n}\right).
\]
\end{theorem}
\begin{proof}
We observe that $\left|M\right|=p^{p\mbox{-}dim\left(M\right)},\left|M^{\perp}\right|=p^{p\mbox{-}dim\left(M^{\perp}\right)}$
and $\left|\mathbb{Z}_{p^{r}}^{n}\right|=p^{r\cdot n}$. By Lemma
\ref{prop:cardinalidade}, we have $p^{p\mbox{-}dim\left(M^{\perp}\right)}=\left|M^{\perp}\right|=\dfrac{\left|\mathbb{Z}_{p^{r}}^{n}\right|}{\left|M\right|}=p^{r\cdot n-p\mbox{-}dim\left(M\right)}$.
Those $p\mbox{-}dim\left(M^{\perp}\right)=r\cdot n-p\mbox{-}dim\left(M\right)$.
\end{proof}
\begin{theorem}[Completion]\label{teo3}
Let $M$ and $N$ be submodules of $\mathbb{Z}_{p^r}^n$ with $N \subset M$,
$p$-$dim\left(M\right)=k$ and $p$-$dim\left(N\right)=t$. Then, each $p$-basis of $N$ can be extended to a $p$-basis of $M$.
\end{theorem}
\begin{proof}
If $M=N$, then the theorem occurs.

If $M\neq N$,  then for $\beta'=\left\{u_1, u_2, \cdots, u_t\right\}$ a $p$-basis of $N$,  there is $v_{1}\in M \setminus N,$ i. e., $v_1\in M$ and $v_1\notin N$. Let $m_1$ be the smallest integer different from $r$, such that $p^{m_1} v_1 \in N$. Thus, $p^{m_1-1}v_1 \not\in N$, i. e., $p^{m_1-1}v_1$ is not a $p$-linear combination of $\beta'.$ Therefore, the set $\beta_1=\left\{p^{m_1-1}v_1\right\}\cup \beta'$ is $p$-linearly independent. In fact, consider $b_1p^{m_1-1}v_1+a_1u_1+\cdots+a_tu_t=0$ with $b_1, a_i \in \mathcal{A}_p, i=1,\cdots,t$. If $b_1=0$, then $a_1u_1+\cdots+a_tu_t=0$ and as $\beta'$ is $p$-linearly independent, it follows that $a_i=0, i=1,\cdots,t,$ and therefore, $\beta_1$ is $p$-linearly independent. If $b_1\neq 0,$ then as $b_1 \in \mathcal{A}_p\setminus\{0\},$ it follows that $p^{m_1-1}v_1=-\left(\dfrac{a_1}{b_1}u_1+\cdots+\dfrac{a_t}{b_1}u_t\right) \in span(\beta')$. By Theorem \ref{teo1}, $span(\beta')$= $p$-$span(\beta')=N$, which is a contradiction, since $v_1 \not\in N.$ In all cases, $\beta_1$ is $p$-linearly independent.

Moreover,  as $p^{m_1}\cdot v_1 \in N$, it follows that $(p^{m_1-1}v_1,u_1, u_2, \cdots, u_t)$ is a $p$-generator sequence. Thus, $\beta_1$ is a $p$-basis.

If $t+1=k$, then $\beta_{1}$ is already a $p$-basis of $M$. Thus,
the proposition is held.

If $t+1<k,$ then there is $v_{2}\in M\setminus p$-$span(\beta_1)$. Let $m_2$ be  the smallest integer other than $r$, such that $p^{m_{2}}v_{2} \in p$-$span(\beta_1)$. By the same argument, $\beta_2=\left\{p^{m_2-1}v_2\right\}\cup\beta_1$ is a $p$-basis. If $t+2=k$, then $\beta_2$ is already a $p$-basis of $M$. Thus, the proposition is held.

Since $k<\infty$, this process is finite and we have the result.
\end{proof}

\begin{corollary}
Let $M$ be a submodule of $\mathbb{Z}_{p^r}^n$ with
$p$-$dim\left(M\right)=k$. Then for each $p$-basis of $Soc(M)$, there is a $p$-basis of $M$ that contains the $p$-basis of $Soc(M)$.
%
%
%
\end{corollary}

\begin{proof}
Consider $N=Soc(M)$ in Theorem \ref{teo3}.

\end{proof}

In the next result, we present a known result for vector spaces that uses the Theorem \ref{teo3} to calculate the dimension of submodules sum.

\begin{theorem}
If $M_1$ and $M_2$ are submodules of $\mathbb{Z}_{p^r}^n$, then $$p\mbox{-}dim(M_1+M_2)=p\mbox{-}dim(M_1)+p\mbox{-}dim(M_2)-p\mbox{-}dim(M_1\cap M_2).$$
\end{theorem}
\begin{proof}
Let us suppose that $W_1\cap W_2 \neq \{\emptyset\}$ and let $\beta=\{w_1,...,w_n\}$ be a $p$-basis of $W_1\cap W_2$. Since $W_1\cap W_2$ is a submodule of $W_1$, by Theorem \ref{teo3}, there is a $p$-basis $\beta_1$ of $W_1$ that contains $\beta$. Thus, let $\beta_1=\{w_1,...,w_n,v_1,...,v_r\}$. Analogously, there is a $p$-basis $\beta_2=\{w_1,...,w_n,u_1,...,u_s\}$ of $W_2$. 

Using the Theorem \ref{teo0}, we will prove that $\gamma=\{w_1,...,w_n,v_1,...,v_r,u_1,...,u_s\}$ is a $p$-basis of $W_1+W_2$.

Let $v\in W_1+W_2$. Then $v=w_1+w_2$ with $w_1\in W_1$ and $w_2\in W_2$. Thus, there are $\alpha_1$,...,$\alpha_{n+r}$ and $\beta_1$,..., $\beta_{n+s}\in \mathcal{A}_p$ such that $$v=\sum_{i=1}^{n}\alpha_iw_i+\sum_{i=1}^{r}\alpha_{n+i}v_i + \sum_{i=1}^{n}\beta_iw_i+\sum_{i=1}^{s}\beta_{n+i}u_i.$$

By rearranging, we have $$ v=\sum_{i=1}^{n}(\alpha_i+\beta_i)w_i+\sum_{i=1}^{r}\alpha_{n+i}v_i + \sum_{i=1}^{s}\beta_{n+i}u_i.$$

Notice that $\alpha_i+\beta_i$ may not belong to $\mathcal{A}_p$. Using the Theorem \ref{teo0} again, there is $\delta_i \in \mathcal{A}_p$ such that  $\sum_{i=1}^n(\alpha_i+\beta_i)w_i=\sum_{i=1}^n\delta_iw_i$. Thus, $\gamma$ is a p-generator set of $W_1+W_2$. Let us prove that $\gamma$ is a $p$-linearly independent set. Consider the $p$-linear combination $$\sum_{i=1}^{n}\delta_iw_i+\sum_{i=1}^{r}\alpha_{n+i}v_i + \sum_{i=1}^{s}\beta_{n+i}u_i=0,$$ where $\delta_i,\alpha_i$ and $\beta_i\in A_p$. Thus $$\sum_{i=1}^{s}\beta_{n+i}u_i=\sum_{i=1}^{n}(-\delta_i)w_i+\sum_{i=1}^{r}(-\alpha_{n+i})v_i.$$ Notice that $-\delta_i$ and $-\alpha_i$ may not belong to $\mathcal{A}_p$, but as $p$-span($w_1$,...,$w_n$) is equal to span($w_1$,...,$w_n$) and $p$-span($v_1$,...,$v_r$) is equal to span($v_1$,...,$v_n$). Then $\sum_{i=1}^{s}\beta_{n+i}u_i \in W_1\cap W_2$.  Therefore, there are $\lambda_1,...,\lambda_n\in \mathcal{A}_p$, such that $$\sum_{i=1}^{s}\beta_{n+i}u_i=\sum_{i=1}^{n}\lambda_{i}w_i,$$ i.e. $$\sum_{i=1}^{s}\beta_{n+i}u_i+\sum_{i=1}^{n}(-\lambda_{i})w_i=0.$$ By rewritting the above equation with elements in $\mathcal{A}_p$, we have $$\sum_{i=1}^{s}\beta_{n+i}u_i+\sum_{i=1}^{n}(\overline{\lambda}_{i})w_i=0, $$ with $\beta_{n+i}, \overline{\lambda}_{i}\in \mathcal{A}_p$. As $\beta_2$ is a $p$-linearly independent set, we have $\beta_{n+i}=0$ for  $0\leq i\leq s$. Thus, $$\sum_{i=1}^{n}\delta_iw_i+\sum_{i=1}^{r}\alpha_{n+i}v_i=0$$ and as $\beta_1$ is a $p$-linearly independent set, we have $\delta_i=0$ for $0\leq i\leq n$ and $\alpha_i=0$ for $0\leq i \leq r.$

In the case $W_1\cap W_2=\emptyset$, $\beta_1$ and $\beta_2$ are $p$-bases of $W_1$ and $W_2$, respectively. Then, $\beta=\beta_1\cup \beta_2$ is a $p$-basis of $W_1+W_2$.
\end{proof}

\end{document}